\documentclass[a4paper,reqno, 12pt]{article}

\usepackage{amsmath,amssymb}
\usepackage{amsthm}
\usepackage{comment}

\usepackage{float}
\usepackage{graphicx}
\usepackage{hyperref}
\usepackage{enumitem}
\usepackage{multicol}

\numberwithin{equation}{section}
\usepackage{url}

\usepackage{color}

\usepackage{tikz,pgfplots}
\tikzset{
every node/.style={circle, inner sep=2pt}
}
\pgfplotsset{yticklabel style={
        /pgf/number format/fixed,
        /pgf/number format/precision=5
},
scaled y ticks=false}

\usetikzlibrary{matrix,arrows,calc}

\pgfplotsset{
    legend entry/.initial=,
    every axis plot post/.code={%
        \pgfkeysgetvalue{/pgfplots/legend entry}\tempValue
        \ifx\tempValue\empty
            \pgfkeysalso{/pgfplots/forget plot}%
        \else
            \expandafter\addlegendentry\expandafter{\tempValue}%
        \fi
    },
}

\newtheorem{theorem}{Theorem}
\newtheorem{lemma}[theorem]{Lemma}

\newtheorem{proposition}[theorem]{Proposition}
\newtheorem{corollary}[theorem]{Corollary}

\theoremstyle{definition}
\newtheorem{definition}[theorem]{Definition}

\newtheorem{remark}[theorem]{Remark}
\newtheorem{example}[theorem]{Example}

\def \G {\mathcal{G}}
\def \CG {\mathcal{CG}}

\newcommand{\diag}{\operatorname{diag}}
\newcommand{\trs}{\operatorname{trs}}
\newcommand{\SNF}{\operatorname{SNF}}
\newcommand{\spec}{\operatorname{spec}}
\newcommand{\invs}{\operatorname{invs}}
\newcommand{\minors}{\operatorname{minors}}

\newcommand{\GSP}{\operatorname{gsp}}

\newcommand{\GIN}{\operatorname{gin}}
\newcommand{\diam}{\operatorname{diam}}
\newcommand{\cof}{\operatorname{cof}}

\pgfplotsset{compat=1.14}
\begin{document}



\title{Distinguishing graphs by their spectra,\\ Smith normal forms and complements}

\author{Aida Abiad\thanks{\texttt{a.abiad.monge@tue.nl}, Department of Mathematics and Computer Science, Eindhoven University of Technology, The Netherlands} \thanks{Department of Mathematics: Analysis, Logic and Discrete Mathematics, Ghent University, Belgium} \thanks{Department of Mathematics and Data Science, Vrije Universiteit Brussel, Belgium} \qquad Carlos A. Alfaro \thanks{\texttt{carlos.alfaro@banxico.org.mx}, Banco de M\'exico, Mexico} \qquad  Ralihe R. Villagr\'an\thanks{\texttt{r.r.villagran.olivas@tue.nl},  Department of Mathematics and Computer Science, Eindhoven University of Technology, The Netherlands}}

\date{}

\maketitle

\begin{abstract}
The search for a highly discriminating and easily computable invariant to distinguish graphs remains a challenging research topic. 
Here we focus on cospectral graphs whose complements are also cospectral (generalized cospectral), and on coinvariant graphs (same Smith normal form) whose complements are also coinvariant (generalized coinvariant). We show a new characterization of generalized cospectral graphs in terms of codeterminantal graphs. We also establish the Smith normal form of some graph classes for certain associated matrices, and as an application, we prove that the Smith normal form can be used to uniquely determine star graphs. Finally, for graphs up to 10 vertices, we present enumeration results on the number of generalized cospectral graphs and generalized coinvariant graphs with respect to several associated matrices.


\end{abstract}


\section{Introduction}

We address the following problem: can we use the spectrum of a graph and its complement (generalized spectrum), or the Smith normal form (SNF) of a graph and its complement (generalized invariants), as a way to distinguish graphs? Finding graph invariants with high discrimination power to distinguish graphs has triggered a lot of work in the last decades. While much attention has been given to the problem of characterizing graphs by their spectra \cite{vDH} and also their generalized spectra \cite{W2017}, less attention has been devoted to understanding which graph properties follow from the SNF. Neither much is known about graph characterizations using the SNF. So far, the SNF has been related to a few structural properties of a graph, such as its spanning trees, sandpile group, the chip-firing game \cite{KlivansBook}, number of (odd) cycles \cite{Grossman}, cuts \cite{Watkins90}, and connectivity. We refer the reader to \cite{stanley} for an overview of the applications of SNF in combinatorics. In fact, Stanley recently commented on the role of the SNF in combinatorics \cite{interview}: \emph{``Although I enjoy these SNF problems, they seem to be mostly problems in algebra, not combinatorics. An exception is the connection between the SNF of the Laplacian matrix of a graph $G$ and chip-firing on $G$... It would be great to have some further combinatorial applications of SNF''}. This provides the initial motivation for our work. 

The SNF has also been related to the spectrum \cite{lorenzini2008,R}. Recently, the SNF of the walk-matrix of a graph has been used to push even further some existing results on characterizations of graphs using their generalized spectra \cite{QWWZ2023}. The SNF has also been proposed as an alternative to the spectrum and has been used to show some graph characterizations \cite{enumeration}.

Let $M$ be a matrix associated with a graph $G$. The \emph{generalized $M$-spectrum} of a graph $G$ is defined to be the pair $(\spec_M(G),\spec_M( \overline{G}))$, where $\overline{G}$ is the complement of $G$. 
Let $G$ and $H$ be two graphs. Then $G$ and $H$ are called \emph{$M$-cospectral} (resp. \emph{generalized $M$-cospectral}) if $G$ and $H$ share the same $M$-spectrum (resp. generalized $M$-spectrum). Recently the generalized spectra have received much attention in the context of graph spectral characterizations, see \cite{W2017} and the references therein.
Two graphs $G$ and $H$ are said to be  $(M,\mathbb{R})$-\emph{cospectral} if they are $(xI+yJ+zM)$-cospectral for every $x,y,z\in \mathbb{R}$, $z\neq 0$. For the adjacency matrix, we can restrict our focus to matrices of the form $yJ-M$ without loss of generality \cite{vDHK}. The following theorem connects the aforementioned concepts for the case when $M$ is the adjacency matrix.

\begin{theorem}[\cite{vDHK,JohnsonN}]\label{original-adjacency}
Let $G$ and $H$ be two graphs. The following statements are equivalent:
    \begin{itemize}
        \item[(i)] $G$ and $H$ are generalized $A$-cospectral,
        \item[(ii)] $G$ and $H$ are $(A,\mathbb{R})$-cospectral, 
        \item[(iii)] $G$ and $H$ are $(yJ-A)$-cospectral for an irrational number $y$, and
        \item[(iv)] $G$ and $H$ are $(yJ-A)$-cospectral for two different values of $y$. 
    \end{itemize}
\end{theorem}

Analogously, we will refer to the multiset of invariant factors of the $\SNF(M(G))$ as the $M$-\emph{invariants} of $G$, and denote it by $\invs_M(G)$.
Therefore, the \emph{generalized $M$-invariants} of a graph $G$ is defined to be the pair $\big(\invs_M(G),\invs_M( \overline{G})\big)$, where $\overline{G}$ is the complement of $G$. Two graphs $G$ and $H$ are said to be $M$-\emph{coinvariant} (resp. \emph{generalized $M$-coinvariant}) if $G$ and $H$ share the same $M$-invariants (resp. $G$ and $H$ share the same generalized $M$-invariants). The concept of coinvariant graphs was introduced in \cite{vince} and further studied in \cite{Codeterminantal, enumeration}. 

In this work, special attention will be devoted to the following matrices associated with graphs: the adjacency ($A$), Laplacian ($L$), signless Laplacian ($Q$), distance ($D$), distance Laplacian ($D^L$), signless distance Laplacian ($D^Q$), transmission-adjacency ($A^{\trs}$), signless transmission-adjacency ($A^{\trs}_+$), degree-distance ($D^{\deg}$), and signless degree-distance ($D^{\deg}_+$) matrices. We will extend results from \cite{enumeration},  \cite{degreedistance}, \cite{ah} and \cite{HS2004} to the framework of generalized $M$-cospectral and generalized $M$-coinvariant graphs.
This paper is structured as follows. 
In Section \ref{sec:relationspectraandocmplements} we observe that Theorem \ref{original-adjacency} also holds for other graph matrices, and in Section \ref{sec:characterization} we prove an extension of this theorem by using codeterminantal graphs. In Section~\ref{sec:SNF} we focus our attention on the recently introduced transmission-adjacency matrices \cite{degreedistance}, and compute their SNF for trees. We will also show that the SNF of such transmission-adjacency matrices can be used for characterizing the graph class of stars (in contrast to the spectrum of $A$, which is known to not be sufficient for such characterization). Moreover, our enumeration results lead us to propose a conjecture which states that almost all trees are determined by their SNF of the aforementioned transmission-adjacency matrices. Finally, in Section \ref{sec:enumeration} we present numerous computational results for graphs up to 10 vertices: we obtain the number of graphs having at least one generalized $M$-cospectral mate (which extends \cite{HS2004}), and we provide the number of graphs having at least one generalized $M$-coinvariant mate (which extends \cite{enumeration}).
Our computational results suggest that the generalized spectrum of transmission-adjacency matrices, and the generalized invariants of the signless distance Laplacian, could be finer invariants to distinguish graphs in cases where other algebraic invariants, such as those derived from the spectrum, fail.


\section{Preliminaries}\label{sec-prel}

Let $G=(V,E)$ be a graph and let $A(G)$, $L(G)$ and  $Q(G)$ be its adjacency matrix, Laplacian matrix, and its signless Laplacian matrix, respectively. 
Let $u,v\in V$ and let $d_G(u,v)$ denote the distance between $u$ and $v$ in $G$. Let $\trs(v)$ denote the \emph{transmission of a vertex} $v$, which is defined as $\trs(v)= \sum_{u\in V} d(u,v)$. Let $\trs(G)$ be the vector of length $|G|$ such that it is $v$-th entry is the transmission of $v$, and let $\deg (G)$ be the vector of length $|G|$ such that its $v$-th entry is the degree of $v$.
We are also interested in the distance matrix $D(G)$, the distance Laplacian matrix $$D^L(G)=\diag (\trs(G))-D(G),$$ the signless distance Laplacian $$D^Q(G)=\diag (\trs(G))+D(G),$$ the adjacency-transmission matrix (along with its signless version) $$A^{\trs} (G)=\diag (\trs(G))-A(G),\quad A^{\trs}_+ (G)=\diag (\trs(G))+A(G), $$ and the degree-distance matrix (along with its signless version) $$D^{\deg}(G)= \diag(\deg_G )-D(G), \quad  D^{\deg}_+ (G) = \diag(\deg_G )+D(G).$$


As usual, $I$, $J$, and $\mathbf{1}$ stand for the identity matrix, the all-ones matrix, and the all-ones vector, respectively ($I_n$, $J_n$ and $\mathbf{1}_{(n)}$ if sizes need to be specified). Let $S\subset [n]$, then $P_{|S}$ denotes the square submatrix of $P$ of size $|S|$ obtained from $P$ by choosing the rows and columns corresponding to $S$ and let $\det(P)_{|S}$ be its determinant.

Let 
$\mathcal{M}=\{A,L,Q,D,D^L,D^Q,D^{\deg},D^{\deg}_+,A^{\trs},A^{\trs}_+ \}.$
Let $G$ and $H$ be two graphs and let $M \in \mathcal{M}$. Recall that $G$ and $H$ are \emph{$M$-cospectral} if they share the same $M$-spectrum. Moreover, $G$ and $H$ are said to be  $(M,\mathbb{R})$-\emph{cospectral} if they are $(yJ-M)$-cospectral for every $y\in \mathbb{R}$. On the other hand, $G$ and $H$ are \emph{generalized $M$-cospectral} if they are $M$-cospectral as well as $\overline{G}$ and $\overline{H}$.
It is well-known that if two graphs $G$ and $H$ are $L$-cospectral if and only if their complement graphs are $L$-cospectral (if and only if $G$ and $H$ are $yJ-L$-cospectral). Furthermore, $G$ and $H$ are generalized $A$-cospectral if and only if $G$ and $H$ are $(A,\mathbb{R})$-cospectral. As we will see in Sections \ref{sec:relationspectraandocmplements} and \ref{sec:characterization}, if $G$, $H$, $\overline{G}$, and $\overline{H}$ have diameter $2$, then this statement also holds for any matrix in $\mathcal{M}$.

\begin{remark}
We should note that the term ``generalized distance spectrum" has been used with a different meaning for merging the $D$-spectrum and the $D^Q$-spectrum via a convex sum \cite{oldGenDistance}, similarly, as it was done for the $A$-spectrum and $Q$-spectrum \cite{mergingAQspec}.
We refer the reader to \cite{Hogbendistsurvey} and \cite{degreedistance} for a study on the properties that the spectra of the matrices (such as $D$, $D^L$, $D^Q$, $D^{\deg}$, and $A^{\trs}$) defined from the distance preserve and on the relationships between their corresponding spectra. 
\end{remark}

Note that if $G$ is a graph and $M\in\{A, L, Q\}$, then we can easily obtain the $M$-spectrum of $\overline{G}$ from $yJ-M$ for some $y\in\mathbb{Z}$. The matrix $yJ-D$ and its minors have been studied in a more general framework in \cite{khare}. However, one cannot obtain the spectrum of $D$ from $yJ-D$ in general. On the other hand, this is indeed possible when $G$ has a diameter 2.

Let $\mathbb{M}_n(\mathcal{R})$ denote the set of square and symmetric matrices of size $n$ with entries on a ring $\mathcal{R}$. Let $A$ and $B$ be two $n\times n$ matrices over some ring with unity $\mathcal{R}$. Recall that $A$ and $B$ are \textit{equivalent} (over $\mathcal{R}$) if there exist unimodular matrices $U$ and $V$ with entries in $\mathcal{R}$ such that $A=UBV$. In particular, two matrices $A$ and $B$ are \textit{similar} (over $\mathcal{R}$) if there exist an unimodular matrix $U$ such that $A=U^{-1}BU$, where $U^{-1}$ is the inverse matrix of $U$. Moreover, if $A$ and $B$ are similar and $U^{-1}=U^T$ we say that $A$ and $B$ are \textit{orthogonally similar}.

Two graphs $G$ and $H$ are called \emph{$M$-coinvariant} if $M(G)$ and $M(H)$ have the same SNF. This property has been characterized by the following folklore result.

\begin{proposition}
    Let $G$ and $H$ be two graphs with $n$ vertices. Then they are $M$-coinvariant if and only if there exist unimodular integer matrices $U,V\in\mathbb{M}_n(\mathbb{Z})$ such that $M(G)=UM(H)V$.
\end{proposition}

It is natural to consider the following extension of coinvariance. Two graphs $G,H$ are said to be \emph{generalized $M$-coinvariant} if and only if they are $M$-coinvariant with $M$-coinvariant complements.

Throughout this paper, we will assume $G$ and $H$ to be two connected simple graphs unless otherwise stated.


\subsection{Relationship among the different (generalized) spectra}\label{sec:relationspectraandocmplements}

In this section, we observe some basic facts on how the spectra of different graph matrices relate to each other, and we will use them to investigate the generality of Theorem \ref{original-adjacency}. Several results in this direction are known. For instance, it is known that if $G$ is $k$-regular ($\deg(G)=k\mathbf{1}$), then its adjacency spectrum determines its (signless) Laplacian spectrum and vice versa \cite{bh}. On the other hand, if $G$ is $t$-transmission regular ($\trs(G) =t\mathbf{1}$) then its distance spectrum determines its (signless) distance Laplacian spectrum \cite{Hogbendistsurvey}. 
Also, if a connected graph is (transmission) regular then the spectrum of $D^{\deg}$ ($A^{\trs}$) is determined by the spectrum of the distance matrix $D$ (adjacency matrix $A$) \cite{degreedistance}.

With the use of a different type of regularity, we can observe the following.

\begin{lemma}\label{lemma:r1}
    Let $s\in \mathbb{N}$ and let $G$ be a graph such that  $\trs(v)+\deg(v)=s$ for every $v\in V(G)$. Then the following statements are satisfied:
    \begin{enumerate}
        \item The spectrum of $A^{trs}$ ($A^{trs}_+$) is determined by the spectrum of $Q$ ($L$) and the converse also holds;
        \item The spectrum of $D^{\deg}$ ($D^{\deg}_+$) is determined by the spectrum of $D^Q$ ($D^L$) and the converse also holds.
    \end{enumerate}
\end{lemma}
\begin{proof}
    This result is a direct consequence of the following equivalences: $A^{\trs}(G)=\diag (\trs(G))-A(G)=sI - \diag (\deg (G))-A(G)=sI-Q(G)$, $A^{\trs}_+(G) =sI-L$, $D^{\deg}(G)= sI- \diag (\trs (G))-D(G)=sI -D^Q(G)$, and $D^{\deg}_+(G) =sI-D^L(G)$.  
\end{proof}

\begin{remark}
    Note that if $\trs(v)+\deg(v)=s$ for every $v\in V(G)$, then $s\geq 2(n-1)$ with equality if and only if $\diam (G)= 2$ or $G$ is complete.
\end{remark}

We also have the following analogous observation.

\begin{lemma}
    Let $s\in \mathbb{N}$ and let $G$ be a graph such that  $\trs(v)-\deg(v)=s$ for every $v\in V(G)$. Then the following statements are satisfied:
    \begin{enumerate}
        \item The spectrum of $A^{trs}$ ($A^{trs}_+$) is determined by the spectrum of $L$ ($Q$) and the converse also holds;
        \item The spectrum of $D^{\deg}$ ($D^{\deg}_+$) is determined by the spectrum of $D^L$ ($D^Q$) and the converse also holds.
    \end{enumerate}
\end{lemma}

\begin{remark}
    Note that if $\trs(v)-\deg(v)=s$ for every $v\in V(G)$, then $s\geq 0$ with equality if and only if $G$ is complete.
\end{remark}

\begin{remark}\label{remark:tregular}
    Note that if $G$ is regular, then $A$, $L$, $Q$, $A^{\trs}$, and $A^{\trs}_+$ mutually determine their respective spectra. On the other hand, if $G$ is transmission regular, then $D$, $D^L$, $D^Q$, $D^{\deg}$, and $D^{\deg}_+$ mutually determine their respective spectra.
\end{remark}

Recall that a pair of graphs is $L$-cospectral if and only if their complement graphs are $L$-cospectral. Thus, it is clear that Theorem \ref{original-adjacency} also holds for the Laplacian matrix.
Note that if $G$ and $H$ are connected, then the spectrum of $yJ-D^L(G)$ is determined by the spectrum of $D^L(G)$. Thus, it is also straightforward to see that Theorem \ref{original-adjacency}  also holds for the distance Laplacian matrix.

\begin{corollary}\label{1stgeneralizedcospectraldistL}
Let $G$ and $H$ be two connected graphs. Then the following statements are equivalent:
\begin{enumerate}
    \item $G$ and $H$ are $D^L$-cospectral,
    \item $G$ and $H$ are $(yJ-D^L)$-cospectral for some $y\in\mathbb{R}$, and
    \item $G$ and $H$ are $(D^L,\mathbb{R})$-cospectral.
\end{enumerate}
\end{corollary}

The next result, which can also be found in \cite{ah2013}, establishes a connection between the spectra of $D^L$ and $L$.

\begin{lemma}\label{lemma:DLL}
    Let $G$ be a graph with diameter at most $2$. Then the spectrum of $D^L$ is determined by the spectrum of $L$, and the converse also holds. In particular, 
    \[\spec_{D_L}(G)=\{0\}\cup \{2n-\lambda_i\}_{i=2}^n \]
    where $\spec_{L}(G)=\{0\}\cup\{ \lambda_i\}_{i=2}^n $.
\end{lemma}
\begin{proof}
    Since $G$ has diameter 2, then $A+D=2(J-I)$ and $\trs (G) + \deg (G) =2(n-1)\bf{1}$. Then $$D^L=
    [2(n-1) I -\diag (\deg (G))]-[2(J-I)-A]=(2n) I-2J-L.$$
    Hence, the spectrum of $D^L$ can be obtained from the spectrum of  $2J+L$ which in turn can be obtained from the spectrum of $L$ (and conversely). Moreover, the second statement holds since $\mathbf{1}$ is an eigenvector for $L$ with eigenvalues $0$ (multiplicity one). Therefore, we can choose a base of eigenvectors corresponding to $\{ \lambda_i\}_{i=2}^n$, which are orthogonal to $\mathbf{1}$ (thus in the kernel of $J$). 
\end{proof}

It is well known that the largest eigenvalue for $L(G)$ is at most $|G|$ with equality if and only if $\overline{G}$ is disconnected, see \cite{bh} for instance. Thus, the second smaller eigenvalue of $D^L(G)$ is at least $n$ with equality if and only if $\overline{G}$ is disconnected, see \cite{ah2013}. Therefore, Lemma~\ref{lemma:DLL} implies the following result which extends \cite[Theorem 2.11]{Rak}.

\begin{corollary}
    Let $G$ be a graph with (diameter at most $2$ and) disconnected complement. Then $G$ is determined by its Laplacian spectrum if and only if it is determined by its distance Laplacian spectrum. 
\end{corollary}

Similarly, we have the following result regarding the spectrum of $D^Q$, $D^{\deg}$, and $Q$.

\begin{lemma}\label{lemma:DQQ}
    Let $G$ be a graph with diameter at most $2$. Then the spectrum of $D^Q$ ($D^{\deg}$) is determined by the spectrum of $2J-Q$ and the converse also holds. Moreover, if $G$ is regular, then $D^Q$ ($D^{\deg}$) is determined by the spectrum of $Q$, and vice versa.
\end{lemma}
\begin{proof}
    The first part of the proof is similar to Lemma~\ref{lemma:DLL}. On the other hand, if $G$ is $k$-regular, then $QJ=JQ=kJ$ and the spectrum of $2J-Q$ is determined by the spectrum of $Q$.
\end{proof}

In particular, if $G$ has diameter at most 2, then $G$ is $k$-regular if and only if is $t$-transmission regular (with $k+t=2(n-1)$). Using this, we obtain the following result.

\begin{proposition}\label{prop:regulardiam2}
    Let $G$ be a regular graph with diameter at most 2 and let $M$ be any matrix in $\mathcal{M}$. Then the $M$-spectrum of $G$ determines the spectrum of any other matrix in $\mathcal{M}$. Moreover, if $H$ is also regular with $\deg(H)=\deg(G)$ and with diameter at most 2, then $G$ and $H$ are $M$-cospectral for some matrix $M$ if and only if they are $M$-cospectral for every matrix in $\mathcal{M}$.  
\end{proposition}
\begin{proof}
    As mentioned above, since $G$ is regular and has diameter at most 2, then it also is transmission regular. We can write this as $AJ=JA$, $DJ=JD$, and $A+D=2(J-I)$. Note that $A$, $D$, and $J$ share $\mathbf{1}$ as an eigenvector. Furthermore, the spectrum of $D$ can be obtained from the spectrum of $A=2(J-I)-D$ since $D$ and $J$ admit a common set of eigenvectors. Conversely, the spectrum of $A$ can be obtained from the spectrum of $D$. This fact, together with Remark~\ref{remark:tregular}, concludes the first part of the proof.
    Thus, given the $M$-spectrum of $G$ for some $M\in\mathcal{M}$, we can determine the spectra of all matrices in $\mathcal{M}$. This implies the second statement since the scaling and translations required for getting the spectra of all matrices for $G$ and $H$ are equivalent.
\end{proof}

It is known that a pair of regular graphs is $A$-cospectral if and only if their complements are $A$-cospectral and similarly for $Q$-cospectral pairs \cite{CRS}.
Thus we can rewrite Proposition \ref{prop:regulardiam2} in terms of generalized $M$-cospectral for $M\in \{A,L,Q\}$, but not for the rest of matrices in $\mathcal{M}$.


\section{A new characterization of generalized cospectral graphs}\label{sec:characterization}

Our main result in this section (Theorem \ref{GC-GxyC}) is a new characterization of generalized cospectral graphs in terms of codeterminantal graphs, which extends Theorem \ref{original-adjacency}. 
To do so, we will need to show a preliminary result (Theorem \ref{cospectral-codeterminantal}), which gives a necessary and sufficient condition for
graphs to be codeterminantal on $\mathbb{Q}[x]$.  

We need some preliminary definitions. Let $\mathcal{R}$ be a ring with unity.

\begin{definition}\label{defdetideals}
    Let $M$ be a matrix with entries in $\mathcal{R}[X]$. The \emph{$k$-th determinantal ideal} of $M$ (denoted by $I_k(M)$) is the ideal generated by $\minors_k(M)$, that is, the set of determinants of every submatrix of size $k\times k$ of $M$.
\end{definition}

In the context of graphs and their corresponding matrices, we have the following definition.

\begin{definition}
    Let $G$ and $H$ be two graphs with $n$ vertices and let $\mathcal{R}$ be a commutative ring with unity. Then $G$ and $H$ are $M^{\mathcal{R}}$-\emph{codeterminantal} if they have the same determinantal ideals with respect to $M$ and $\mathcal{R}$ for every $1\leq k\leq n $.
\end{definition}

The following result provides a useful way to determine if a pair of matrices are codeterminantal.
 
\begin{proposition}{\cite{Nbook}}\label{prop:equiimpliescodeterminantal}
If $M$ and $N$ are equivalent matrices, then 
$M$ and $N$ are codeterminantal.
\end{proposition}



Next, we prove a new connection between cospectral and codeterminantal graphs.
In \cite[Theorem 33]{Codeterminantal}, this connection was actually shown for $\mathbb{R}[x]$. Here we prove that, in fact, this can be extended to $\mathbb{Q}[x]$.

\begin{theorem}\label{cospectral-codeterminantal}
Two graphs $G$ and $H$ are $M$-cospectral if and only if they are $(xI-M)^{\mathbb{Q}}$-codeterminantal.
\end{theorem}
\begin{proof}
Let $G$ and $H$ be two $M$-cospectral graphs. Since $M$ is symmetric, then there exists an orthogonal matrix $U$ such that $M(G) = U^T M(H)U$. In particular, $xI-M(G)$ and $xI-M(H)$ are equivalent over $\mathbb{R}$, 
then by Proposition~\ref{prop:equiimpliescodeterminantal} we have that $xI-M(G)$ and $xI-M(H)$ are codeterminantal over $\mathbb{R}[x]$. Furthermore, since all the coefficients of the generating polynomials of any determinantal ideal of $xI-M$ are integers, then any relation defined by the generators of a determinantal ideal of $xI-M$ over $\mathbb{R}$ can be characterized by a set of rational numbers. Hence, we can conclude that a set of polynomials defining a base for $I_k^{\mathbb{R}}(xI-M)$ is also a base for $I_k^{\mathbb{Q}}(xI-M)$, and $G$ and $H$ are $M_x^{\mathbb{Q}}$-codeterminantal. The other direction follows by Definition~\ref{defdetideals}, since the characteristic ideal is encoded in the $n$-th determinantal ideal of $xI-M$.
\end{proof}

\begin{remark}
Note that Theorem~\ref{cospectral-codeterminantal} holds for every graph matrix defined in $\mathcal{M}$. On the other hand, for graph matrices with (non-integer) real entries (like the normalized adjacency matrix from \cite{chung}), 
the analogous holds replacing $\mathbb{Q}$ with $\mathbb{R}$.
\end{remark}

Next, we state a last lemma that we will use later on.

\begin{lemma}{\cite{Graham}}\label{lemma:jn}
    Let $M\in \mathbb{M}_n (\mathcal{R})$, and $y$ an indeterminate that commutes with $\mathcal{R}$. Then $$\det (M+yJ) =\det(M)+y\cof(M)$$
    where $\cof(M)=\sum_{i,j\in [n]} (-1)^{i+j} \det(M)_{i,j}$ and $(M)_{i,j}$ is the submatrix of size $(n-1)\times (n-1)$ obtained from $M$ by deleting the $i$-th row and the $j$-th column. 
\end{lemma}
 
We are now ready to state the main result of this section.

\begin{theorem}\label{GC-GxyC}
Let $M\in \mathcal{M}=\{A,L,Q,D,D^L,D^Q,D^{\deg},D^{\deg}_+,A^{\trs},A^{\trs}_+ \}$, and let us denote 
$M_{x,y}= xI+yJ-M$. 
Let $G$ and $H$ be two graphs (connected if $M$ is not the adjacency or the (signless) Laplacian). Then the following statements are equivalent:
\begin{itemize}
    \item[(i)] $G$ and $H$ are $(yJ-M)$-cospectral for two different values of $y$. 
    \item[(ii)] $G$ and $H$ are $(M,\mathbb{R})$-cospectral.
    \item[(iii)] $G$ and $H$ are $M_{x,y}^{\mathbb{Q}}$-codeterminantal.
    \item[(iv)] $G$ and $H$ are $(yJ-M)$-cospectral for an irrational number $y$.
\end{itemize}
Moreover, if $M\in\{A,L,Q\}$ or $G$, $\overline{G}$, $H$ and $\overline{H}$ have diameter 2, then the following is also an equivalent statement:
\begin{itemize}
    \item[(v)] $G$ and $H$ are generalized $M$-cospectral. 
\end{itemize}
\end{theorem}

\begin{proof} \leavevmode
\begin{description}
\item[$(i)\iff (ii)$] Clearly (ii) implies (i). Conversely, if $y_1\neq y_2$ and $G$ and $H$ are $y_i-M$-cospectral for $i=1,2$, then $$\det (xI+M(G)-y_iJ) = \det (xI+M(H)-y_iJ)\text{ for }i=1,2.$$ Thus, By Lemma~\ref{lemma:jn},
$$\det (xI+M(G)) -y_i \cof (xI+M(G)) = \det (xI+M(H)) -y_i\cof(xI+M(H))$$
for $i=1,2.$
Since $y_1\neq y_2$, $\cof(xI+M(G))=\cof(xI+M(H))$ and also the characteristic polynomials of $G$ and $H$ are equal. Hence, $G$ and $H$ are $(yJ-M)$-cospectral for any value of $y$.
\item[$(ii) \Rightarrow (iii)$] Following an analogous argument as in \cite[Theorem 2]{JohnsonN} but using $M$, it is shown that there is an orthogonal matrix $U$ such that $U^T JU=J$ and $M(G) = U^T M(H) U$. 
We have that $yJ-M(G)$ is orthogonally similar to $yJ-M(H)$ via $U$ for every $y\in\mathbb{R}$. Then, by Theorem \ref{cospectral-codeterminantal}, $G$ and $H$ are $M_{x,y}^{\mathbb{Q}}$-codeterminantal.
\item[$(iii) \Rightarrow (i)$] It is straightforward by evaluation on $(x,y)$ and the fact that the corresponding characteristic polynomials are encoded in the $|G|$-th determinantal ideals.
\item[$(iv)\Rightarrow (i)$] An analogous argument to the one presented in the proof of \cite[Theorem 2]{vDHK} shows that if $G$ and $H$ are $(yJ-M)$-cospectral for only one value of $y$ then $y$ is rational. On the other hand, it is clear that $(ii)$ implies $(iv)$. Thus we have the equivalence for the first four statements.
\end{description}
Now, to conclude this proof, we will prove that $(v)$ implies $(i)$ and $(ii)$ implies $(v)$. In order to do so, we only need to show that, under the given assumptions, $M(\overline{G})=M_{x,y}$ for some $x,y\in\mathbb{R}$ such that $(x,y)\neq (0,0)$. First, let us notice that for any graph
\begin{align*}
    &A(\overline{G})=(J-I)-A(G)=A(G)_{x=-1,y=1},\\
    &L(\overline{G})=(nI-J)-L(G)=L(G)_{x=n, y=-1},\text{ and}\\
    &Q(\overline{G})=(n-2)I+J-Q(G)=Q(G)_{x=n-2, y=1}.
\end{align*}
Furthermore, for graphs with diameter $2$ such that their complements also have diameter $2$, we have that $D(G) = 2(J-I) -A(G)$ and $D(\overline{G})= (J-I)+A(G)$. Therefore $D(\overline{G})=(-3)I+3J-D(G)=D(G)_{x=-3,\, y=3}$. Moreover $\diag (\trs(G)) + \diag(\trs(\overline{G})) = (3n-3) I$. Thus,
\begin{align*}
D^L(\overline{G}) &=D^L(G)_{x=3n,\,y=-3},\\
D^Q (\overline{G} ) &= D^Q (G) _{x=3n-6,\, y=3},\\
A^{\trs} (\overline{G})&=A^{\trs}(G)_{x=3n-2,\, y= -1},\\ 
D^{\deg} (\overline{G})&=D^{\deg} (G) _{x=n+2,\, y = -3}, \\
A^{\trs}_+ (\overline{G})&=A^{\trs}_+(G)_{x=3n-4,\, y= 1},\\ 
D^{\deg} (\overline{G})&=D^{\deg}_+ (G) _{x=n-4,\, y = 3}.
\qedhere\end{align*}
\end{proof}

Note that the extra condition $(v)$ in Theorem~\ref{GC-GxyC} can be restated for graphs such that for any two vertices $u$ and $v$ in $G$ we have that $d_G (u,v) + d_{\overline{G}} (u,v) = k$ for some positive integer $k$. However, this condition implies that $G$ and its complement have diameter at most 2. Thus, $k=3$ and condition $(v)$ are equivalent.

Finally, note that for pairs of graphs with diameter 2 and complements of diameter 2, the above theorem is reduced to the following corollary.

 \begin{corollary}\label{coro:complementdiam2ALQ}
Let $G$ and $H$ be two graphs such that $G$, $H$, $\overline{G}$, and $\overline{H}$ have diameter 2. Then the following statements hold:
\begin{enumerate}
    \item $G$ and $H$ are generalized $A$-cospectral if and only they are generalized $D$-cospectral,
    \item $G$ and $H$ are generalized $L$-cospectral if and only if they are generalized $M$-cospectral for any $M\in \{D^L, D^{\deg}_+, A^{\trs}_+ \}$, and
    \item $G$ and $H$ are generalized $Q$-cospectral if and only if they are generalized $M$-cospectral for any $M\in \{D^Q, D^{\deg}, A^{\trs} \}$.
\end{enumerate}
\end{corollary}

\begin{proof}
Under these hypotheses it is easy to see that a pair of graphs is generalized $A$-cospectral if and only if they are generalized $D$-cospectral. Moreover, recall that for diameter 2, by Lemmas~\ref{lemma:r1} and Lemma~\ref{lemma:DLL}, we have that the spectrum of any matrix in $\{L,A^{\trs}_+, D^L, D^{\deg}_+ \}$ can be obtained from the spectrum of any other such matrix.
Furthermore the spectrum of $Q$ and $A^{\trs}$ determine each other, as well as for $D^{\deg}$ and $D^Q$. Lastly, by Lemma~\ref{lemma:DQQ}, the spectrum of $D^{\deg}$ ($D^Q$) is determined by (and determines) the spectrum of $2J-Q$. Hence, the result follows by Theorem~\ref{GC-GxyC}.
\end{proof}

\section{A new application of the Smith normal form to combinatorics}\label{sec:SNF}

In this section, we find new combinatorial applications of the SNF, in particular regarding the graph characterization problem.
With the help of algebraic tools (determinantal ideals), we will obtain the SNF of several graph classes (trees) for some associated matrices. Moreover, we will show a new application of the SNF for graph characterizations; in particular, we will show that star graphs are determined by the SNF of its $A^{\trs}$ matrix and by the SNF of its $A^{\trs}_+$ matrix.

\subsection{Previous results}\label{sec:previousSNF}

The Smith normal form (SNF) has been studied for several graph classes and different associated matrices. For the Laplacian matrix of a graph $G$, the subject of the SNF relates to the \textit{sandpile group}, the torsion part of its cokernel ($\mathbb{Z}^{|V|}/ \text{Im} (L(G))$), denoted by $K(G)$, which in turn is an algebraic invariant for the chip-firing game, therefore the literature is vast, we refer the reader to \cite{KlivansBook}. Similarly, the cokernel of the adjacency matrix, which is known as the \textit{Smith group}, and the SNF of adjacency matrices, have been computed for several graphs in the literature.  For the distance matrix, the SNF has been determined for cycles, complements of cycles, trees, unicyclic graphs, wheels, and wheel graphs with trees attached to their vertices and partially determined for complete multipartite graphs, see \cite{BK} and \cite{HW2008}. In particular, star graphs are known to be determined by their SNF for the distance Laplacian matrix. Moreover, complete graphs are also determined by the SNF's of their (signless) distance Laplacian, (signless) adjacency-transmission, and (signless) degree-distance matrices, see \cite{Codeterminantal, enumeration} and \cite{degreedistance}.

The SNF has also been shown to be particularly useful to distinguish graphs. The SNF of the complete graph is known for the adjacency (distance), and the (signless) Laplacian matrix, see \cite{enumeration} (which in this case, and since $A(K_n)=D(K_n)$, it also corresponds to the (signless) distance Laplacian,  the (signless) transmission-adjacency matrix, and the (signless) degree-distance matrix). Furthermore, it is known that complete graphs are determined by the SNF of any matrix in $\{L,Q,D^L,D^Q,A^{\trs}, A^{\trs}_+,D^{\deg},D^{\deg}_+\}$ \cite{2trivialAlfVal,degreedistance}. Note that complete graphs are not determined by the SNF of its adjacency (distance) matrix, since the SNF of $A(K_5)=D(K_5)$ is equal to the SNF of the adjacency matrix of the butterfly graph, $K_1\vee 2K_2$, and the SNF of the distance matrix of $K_5-e$. Namely, $\SNF(A(K_5))=\diag (1,1,1,1,4) = \SNF (A(K_1\vee 2K_2))=\SNF (D(K_5-e))$.

\subsection{A new graph characterization using the Smith normal form}\label{sec:newSNF}

Let $S_{n+1}:=K_{n,1}$ denote the graph class of stars. In \cite{degreedistance}, the SNF of star graphs for $D^{\deg}$ and $D^{\deg}_+$ is studied. Moreover, in \cite{Codeterminantal} it was shown that star graphs are determined by the SNF of $D^L$. Here, we will establish the SNF of star graphs for $A^{\trs}$ and $A^{\trs}_+$. 

\begin{proposition}\label{SNFstar}
    Both the SNF of $A^{\trs}(S_{n+1})$ and the SNF of $A^{\trs}_+(S_{n+1})$ are equal to
    $$\diag (1,1,2n-1,\ldots ,2n-1,2n(n-1)(2n-1)).$$
\end{proposition}

We will skip the proof for the moment to focus on trees in general. It is well known that almost all trees are cospectral (\cite{Schwenk}). Regarding the SNF, all trees are $L$-coinvariant, $Q$-coinvariant, and $D$-coinvariant (\cite{HW2008}). On the other hand, it is also conjectured that almost all
trees are determined by the SNF of their $D^L$, by the SNF of their $D^Q$ (see \cite{enumeration, ah}), by the SNF of their $A^{\trs}_+$, and by the SNF of their $A^{\trs}$ (\cite{degreedistance}).

We recall the following lemma that we will need later on and which provides a useful way of computing the Smith normal form of a matrix.   

\begin{lemma}[Elementary divisors theorem, \cite{Jacobson}]\label{lemma:elementarydivisors}
    Let $M$ be an $n\times m$ matrix with entries in a principal ideal domain (PID). Then the $k$-th invariant factor of $M$ is equal to $\Delta_k / \Delta_{k-1}$, where $\Delta_k$ is the greatest common divisors of the minors of size $k$ of $M$, and $\Delta_0=1$.
\end{lemma}

Given a graph $G$, the $k$-th \textit{critical ideals}, that is, the $k$-th determinantal ideals of $L(G,X)=\diag (x_1,x_2,\ldots ,x_n)-A(G)$ have been extensively studied. In particular, they have been described for trees, see \cite{citrees}. Let us recall that a pair of graphs are $M$-coinvariant if and only if they are $M$-codeterminantal over $\mathbb{Z}$ (\cite[Theorem 35]{Codeterminantal}). In particular, by \cite[Proposition 14]{Codeterminantal}, the $k$-th critical ideal of $G$ over the integers, $I^{\mathbb{Z}}_k(L(G,X))$, evaluated at $X=\trs(G)$ (respectively $(-1)\trs(G)$) is generated by $\Delta_k(A^{trs}(G))$ (resp. $\Delta_k(A^{trs}_+(G))$), the greatest common divisor of the $k$-minors of $A^{\trs}(G)$ ($A^{\trs}_+(G)$).
We will apply these results to the SNF of $A^{\trs}$ and the SNF of $A^{\trs}_+$ for trees. 
Let $T$ be a tree and $P^T_{u,v}$ be the set of edges of the only path from $u$ to $v\neq u$ in $T$. 
Moreover if $U\subseteq V(T)$, let
\[\rho (U) = |\{ F\subset E(T) : |F|=|U|-1,\ P_{u,v}^T \cap F \neq \emptyset\ \forall\ u, v\in U  \}|.\]
Now, let $T^l$ be the non-simple graph resulting from adding a loop to each vertex of $T$. A set of edges $\mu\in E(T^l)$ is called a 2-matching of $T^l$ if every vertex has at most two incidences with $\mu$ (note that a loop counts precisely as two incidences). Furthermore, a 2-matching of $T^l$ is called minimal if it is minimal with respect to the number of loops in $\mu$, and the set of vertices with loops in $\mu$ is denoted by $l(\mu)$. Finally, let $2\text{Mat}_{\min}^k(T^l)$ be the set of minimal 2-matchings of size $k$ of $T^l$.
Hence, we obtain the following results for the SNF of $A^{\trs}$ of trees.

\begin{theorem}\label{SNFtrs-trees}
    Let $T$ be a tree with $n\geq 3$ vertices. Then the SNF of $A^{\trs}(T)$ ($A^{\trs}_+(T)$) is equal to
    $$\diag \left(1,1,\Delta_3, \frac{\Delta_4}{\Delta_3}, \ldots,\frac{\Delta_n}{\Delta_{n-1}} \right),$$
    where $\Delta_k = \gcd ( \{ \det(A^{\trs}(T))_{|l (\mu) }: \mu \in 2\text{Mat}_{\min}^{k}(T^l) \} )$. In particular, 
    $$\Delta_n = \sum\limits_{a=1}^n\sum\limits_{U\subseteq T \atop |U|=a} \rho (U)\prod\limits_{ U=\{ u_{i_1}\ldots ,u_{i_a}\} } c(u_{i_j}), $$
    where $c(u) = \trs(u)-\deg(u)$ for any vertex $u\in T$.
\end{theorem}
\begin{proof}
    The result for $A^{\trs}(T)$ follows from \cite[Theorem 3.7]{citrees} (by evaluating at $X=\trs(G)$) and Lemma~\ref{lemma:elementarydivisors}. In particular, we can compute $\Delta_n$ by counting the spanning trees of the multigraph $H=T+v$, where $v\notin T$ is a new vertex with $c(u)=\trs(u)-\deg(u)$ edges between $u$ and $v$ for every vertex $u\in T$. Now, to see that this is also the SNF of $A^{\trs}_+(T)$ it is sufficient to show that if $F$ is a forest, then $|d(F)_{|U}|=|d_+(F)_{|U}|$ for every $U\subseteq V(F)$, where $d(F)_{|U}=\det(L(F, X))_{|U}$ and $d_+(F)_{|U}=\det(\diag (x_1,\ldots,x_n)+A(F))_{|U}$ (for brevity).
    We proceed by induction on $|F|$. It is easy to see that it holds for $|F|=3$. Now, assume that $|d(F^{'})_{|U^{'}}|=|d_+(F^{'})_{|U^{'}}|$ for every forest with $1\leq |F^{'}|\leq n-1$ and every $U^{'}\subseteq F^{'}$. Thus, let $F$ be a forest with $n$ vertices and let $U\subseteq F$ and $H$ be the induced subgraph of $F$ spanned by $U$. Moreover, let $u\in U$ be a vertex of $F$ in $U$ such that $\deg_H(u)\leq 1$ (if $\deg_H(u)= 1$ let $(u,w)\in E(H)$). Therefore, by Laplace expansion and our induction hypothesis, if $\deg_H(u)=1$ we have that
    \begin{align*}
    |d_+(F)_{|U} |&=|x_ud_+(F)_{|U\setminus \{ u\}}-d_+(F)_{|U\setminus \{ u,w \} }|= |x_ud_+(F-u)_{|U\setminus \{ u\}}-d_+(F-u)_{|U\setminus \{ u,w \} }|\\
        &=|x_ud(F-u)_{|U\setminus \{ u\}}-d(F-u)_{|U\setminus \{ u,w \} }|=|x_ud(F)_{|U\setminus \{ u\}}-d(F)_{|U\setminus \{ u,w \} }|\\
        &=|x_ud(F)_{|U\setminus \{ u\}}-(-1) \det((L(F, X))_{|U})_{u,w} =|d(F)_{|U}|,
 \end{align*}

    for every $U\subseteq V(F)$. Similarly, if $\deg_H(u)=0$, then
    \begin{align*}
        |d_+(F)|_{U} |&= |x_ud_+(F-u)|_{U\setminus \{ u\} }|=|x_ud(F-u)|_{U\setminus \{ u\}}|\\
        &=|x_ud(F-u)|_{U\setminus \{ u\}}|=|x_ud(F)|_{U\setminus \{ u\}}|\\
        &=|d(F)|_U|,
    \end{align*}


    for every $U\subseteq V(F)$. This finishes the proof.
\end{proof}

This result may shed some light on whether trees are determined by SNF of $A^{\trs}$ ($A^{\trs}_+$), see \cite{degreedistance}. In particular, we can see that the star graph with $n$ leaves ($S_{n+1}$) has
$$\Delta_k=(2n-1)^{k-2},$$
for $3\leq k\leq n$. Moreover, if $v_{n+1}$ is the dominant vertex then $c(v_{n+1})=0$ and if $v_{n+1}\notin U$ then $\rho(U)=|U|$. Lastly, $c(u)=2n-2$ for any leaf. Therefore
$$\Delta_n= \sum\limits_{a=1}^{n}\sum\limits_{U\subseteq T\setminus \{v_{n+1}\} \atop |U|=a} a  (2n-1)^a=\sum\limits_{a=1}^{n} \binom{n}{a}a(2n-1)^a=2n(n-1)(2n-1)^{n-1},$$
which proves Proposition~\ref{SNFstar}.

Furthermore, using some results from \cite{2trivialAlfVal}, we  obtain the following theorem.

\begin{theorem}
    Star graphs are determined by the SNF of $A^{\trs}$ ($A^{\trs}_+$).
\end{theorem}
\begin{proof}
    First, we know that star graphs have at most two invariant factors equal to $1$ for $A^{\trs}$ ($A^{\trs}_+$) and, by \cite[Theorem 4.3]{2trivialAlfVal}, any such graph is an induced graph of a tripartite complete graph or an induced subgraph of $(mK_1)\vee (K_p+K_q)$. Note that every such (connected) graph has diameter $2$ and that its transmission vector is easy to compute. For instance if $G=K_{m,p,q}$ with $n=m+p+q$, $m,p,q\geq 2$ and respective maximal independent sets $V_m$, $V_p$ and $V_q$. Then, if $v\in V_m$, $\trs (v)=p+q-2(m-1)$ and similarly if $v\in V_p, V_q$. Thus, evaluating in \cite[Theorem 4.7]{2trivialAlfVal} gives us that the third invariant factor of $A^{\trs}(G)$ ($A^{\trs}_+(G)$) is $\gcd(2,p+q-2(m-1),m+q-2(p-1),m+p-2(q-1))$ which is different than $2n-1=2(m+p+q-1)-1$ independently of the values of $(m,p,q)$. Continuing in a similar manner with all possible choices of $m$, $p$ and $q$ and for $(mK_1)\vee (K_p+K_q)$ using \cite[Theorem 4.8]{2trivialAlfVal} we have that the star graphs are determined by their SNF with respect to $A^{\trs}$ ($A^{\trs}_+$).
\end{proof}

\begin{remark}
    Note that through the Smith normal form of $A^{\trs}(G)$ we can obtain the algebraic characterization of the sandpile group of a supergraph of $G$. Namely, let $H$ be a graph with vertex set $V(H)=V(G)\cup \{h\}$, $h\notin G$, and such that $E(H)$ is $E(G)$ together with $\trs(v)-\deg(v)$ copies of $(v,h)$ for each $v\in V$. If $\SNF (A^{\trs}(G))=\diag(d_1,\ldots,d_{|G|})$ and $d_k$ is the first invariant factor different than $1$, then $K(H)\cong \mathbb{Z}_{d_k}\oplus\cdots\oplus\mathbb{Z}_{d_{|G|}}$.
    Moreover, if $G=T$ is a tree, then $H$ (embedded in the plane) is the plane dual graph of certain bi-connected outerplane graph $H^*$. Thus, by \cite{Cori}, $K(H^*)\cong K(H)$ (see \cite{outerplane} for more details on sandpile groups of outerplane graphs).
\end{remark}

Before concluding this section, we will compute the SNF of the (signless) transmission adjacency matrix for another family of graphs with diameter 2. Namely, the complete multipartite transmission regular graphs. The critical ideals for complete multipartite graphs were studied in \cite{yibogao}. Thus we can use the results therein to compute the SNF for the aforementioned matrices. Let $C_k(G)$ be the $k$-th determinantal ideal of $xI-A(G)$, then the ideals $C_k(G)$ are usually called the \textit{characteristic ideals} of $G$. In particular, we will use the following result.

\begin{lemma}\cite[Corollary 4.1]{yibogao}\label{lemma:gao}
    Let $G$ be a complete multipartite graph with $m\geq 2$ parts of size $s\geq 2$. Then
    \[C_k(G))=\begin{cases}
        \langle 1\rangle & \quad 1\leq k\leq m-1,\\
        \langle (m-1)x^{j-m}x^{j-m+1}\rangle &\quad m\leq k\leq ms-m,\\
        \langle x^{j-m+1}(x+s)^{m+j-ms-1},&  \\
        \sum_{a=0}^{m+j-ms}(ms-j-1+a)s^a\binom{m+j-ms}{a}x^{2j-ms-a} \rangle & \quad ms-m < k < ms,\\
        \langle \sum_{a=0}^m(a-1)r^a\binom{m}{a}x^{ms-a} \rangle & \quad k=ms.
    \end{cases}\]
\end{lemma}
    
If a complete multipartite graph $G$ with $m\geq 2$ parts of size $s_1,\ldots ,s_m\geq 2$ is $t$-transmission regular, then $s_i=s_j$ ($=s$) for every $1\leq i,j\leq m$ and $t=n+s-2=ms+s-2$.

\begin{proposition}\label{SNFtrs-multipartite}
    Let $G$ be a $t$-transmission regular complete multipartite graph with $m\geq 2$ parts of size $s\geq 2$.  Then the SNF of $A^{\trs}(G)$ is equal to
\begin{align*}
&\diag(\mathbf{1}_{(m-1)}, a, t \mathbf{1}_{(ms-2m)}, 2t, t(t+s) \mathbf{1}_{(m-2)}, \frac{(s-1)t(t+s)}{a})\qquad \text{ if }m, s= 0\text{ mod } 2,\\
&\diag(\mathbf{1}_{(m-1)}, a, t \mathbf{1}_{(ms-2m)}, t, t(t+s) \mathbf{1}_{(m-2)}, \frac{2(s-1)t(t+s)}{a}) \qquad \text{ otherwise,}
\end{align*}
where $a=\gcd (m-1,2(s-1))$.


    
    
\end{proposition}
\begin{proof}
    Let $\Delta_k$ be the greatest common divisor of the minors of $A^{\trs}(G)$. By evaluation in Lemma~\ref{lemma:gao} we have that $\Delta_j =1$ for every $1\leq j\leq m-1$ and $\Delta_m=\gcd(m-1,t)=\gcd(m-1,s(m-1)+2(s-1))=\gcd(m-1,2(s-1))=a$. Furthermore if $m+1\leq j\leq ms-m$, then $\Delta_j=\gcd((m-1)t^{j-m},t^{j-m+1})=t^{j-m}a$. Now, if $ms-m+1\leq j\leq ms-1$ then 
    \begin{align*}
        \Delta_j&=(r-1)t^{j-m}(t+r)^{j+m-1-mr}\gcd(t,r(m-1))\\
        &=(r-1)t^{j-m}(t+r)^{j+m-1-mr}\gcd(r(m-1),2(r-1)).
    \end{align*}
    Finally, $\Delta_n=2(r-1)t^{ms-m}(t+r)^{m-1}$. Let $b=\gcd(r(m-1),2(r-1))$, note that the results now follows since $\frac{b}{a}=2$ if and only if both $m$ and $s$ are even (and $a=b$ otherwise). 
\end{proof}

An analogous argument as used in the proof of Proposition \ref{SNFtrs-multipartite} can be used to obtain the SNF of $A^{\trs}_+$ for complete multipartite regular graphs.

\begin{proposition}
    Let $G$ be a $t$-transmission regular complete multipartite graph with $m\geq 2$ parts of size $s\geq 2$. Then the SNF of $A^{\trs}_+(G)$ is as equal to
\begin{align*}
&\diag(\mathbf{1}_{(m-1)}, a, t \mathbf{1}_{(m(s-2))}, 2t, t(t-s) \mathbf{1}_{(m-2)}, \frac{(ms-1)t(t-s)}{a} ) \qquad \text{ if }m,s=0\text{ mod }2,\\
&\diag(\mathbf{1}_{(m-1)}, a, t \mathbf{1}_{(ms-2m)}, t, t(t-s) \mathbf{1}_{(m-2)}, \frac{2(ms-1)t(t-s)}{a} )\qquad \text{ otherwise,}
\end{align*}
where $a=\gcd (m-1,2(s-1))$.
\end{proposition}

Finally, we briefly mention the differences between the spectrum and the SNF when it comes to deducing graph combinatorial properties.

\begin{example}
For any graph $G$, the SNF of the Laplacian matrix $L$ preserves the determinant of its minors of size $|G|-1$, that is, its number of spanning trees (therefore being a tree or a complete graph, for instance, is determined) \cite{KlivansBook}. It also preserves the number of connected components \cite{Doob}, and the algebraic structure of its sandpile group \cite{KlivansBook}. On the other hand, it does not preserve the number of edges (which is preserved by the spectrum of $L$ and $A$). In fact, the smallest pair of graphs with the same SNF for $A$ but a different number of edges is the cycle and star graphs on $4$ vertices. Moreover, the smallest pair of graphs for which the number of edges is different but they share Laplacian SNF are the triangular prism (9 edges) and the triangular bipyramid ($H$) with a pendant vertex glued at any vertex of $H$ (10 edges). 
Figure~\ref{fig:1} shows the only pair of $L$-coinvariant graphs with at most 7 vertices having a different number of edges, vertex connectivity, and edge connectivity (and thus they are not $L$-cospectral nor $A$-cospectral). Their invariant factors for the Laplacian matrix are $(1,1,1,1,12,60)$. Indeed, in \cite{enumeration} we can see that, up to 10 vertices, very few graphs are determined by their SNF for $M=\{A,L,Q,D\}$, which might indicate that, in general, not many combinatorial properties may be preserved by the SNF of such matrices.

\begin{figure}[ht]
    \centering
    \begin{tabular}{cc}
    \begin{tikzpicture}[scale=1,thick]
		\tikzstyle{every node}=[minimum width=0pt, inner sep=2pt, circle]
			\draw (-1.5,0) node[draw] (0) { \tiny 0};
			\draw (0,0) node[draw] (1) { \tiny 1};
			\draw (1.5,0) node[draw] (2) { \tiny 2};
			\draw (-0.75,-1) node[draw] (3) { \tiny 3};
			\draw (0.75,-1) node[draw] (4) { \tiny 4};
			\draw (-0.75,1) node[draw] (5) { \tiny 5};
			\draw (0.75,1) node[draw] (6) { \tiny 6};
			\draw  (5) edge (6);
			\draw  (0) edge (3);
			\draw  (1) edge (3);
			\draw  (2) edge (3);
			\draw  (2) edge (4);
			\draw  (1) edge (4);
			\draw  (0) edge (4);
			\draw  (1) edge (6);
			\draw  (2) edge (6);
			\draw  (0) edge (5);
			\draw  (1) edge (5);
			\draw  (2) edge (5);
			\draw  (0) edge (6);
		\end{tikzpicture}     &  
		\begin{tikzpicture}[scale=1,thick]
		\tikzstyle{every node}=[minimum width=0pt, inner sep=2pt, circle]
			\draw (0,1) node[draw] (0) { \tiny 0};
			\draw (-0.755,0) node[draw] (1) { \tiny 1};
			\draw (-0.75,-1) node[draw] (2) { \tiny 2};
			\draw (0.75,0) node[draw] (3) { \tiny 3};
			\draw (0.75,-1) node[draw] (4) { \tiny 4};
			\draw (1.5,0) node[draw] (5) { \tiny 5};
			\draw (-1.5,0) node[draw] (6) { \tiny 6};
			\draw  (0) edge (4);
			\draw  (0) edge (3);
			\draw  (0) edge (1);
			\draw  (0) edge (2);
			\draw  (3) edge (5);
			\draw  (4) edge (5);
			\draw  (1) edge (6);
			\draw  (2) edge (6);
			\draw  (1) edge (3);
			\draw  (1) edge (2);
			\draw  (1) edge (4);
			\draw  (2) edge (3);
			\draw  (2) edge (4);
			\draw  (3) edge (4);
					\end{tikzpicture}\\
			$G_1$ & $G_2$
    \end{tabular}
    \caption{Smallest $L$-coinvariant pair on $7$ vertices with a different number of edges, vertex-connectivity, and edge-connectivity.}
    \label{fig:1}
\end{figure}
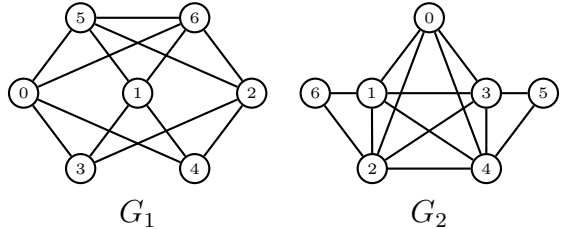

\end{example}

\section{Enumeration results}\label{sec:enumeration}

In this section, we present computational results on the amount of generalized coinvariant and generalized cospectral connected graphs up to $10$ vertices, for numerous matrices associated with graphs. Moreover, we also provide several similar enumeration results up to $11$ vertices for graphs with diameter 2 such that their complement has diameter 2 as well.

Let $\GSP_ n (M)$ denote the \emph{generalized $M$-cospectral uncertainty}, defined as the ratio between the number of graphs with at least a generalized $M$-cospectral mate and the number of connected graphs (with connected complement in the case of distance defined matrices). Similarly, let $\GIN_ n (M)$ denote the \emph{generalized $M$-coinvariant uncertainty}.

Let $\mathcal{G}_n^{\GIN}(M)$ be the set of connected graphs $G$ with $n$ vertices such that $G$ has at least one connected
generalized $M$-coinvariant mate for a given graph matrix $M$. 
Let $\G_n$ be the set of connected simple graphs with $n$ vertices and let $\CG_n\subsetneq \G_n$ be the set of connected graphs such that their complement is also connected. Note that the distance is well-defined for every pair of vertices on a graph $G$ (and thus the transmission of every vertex) only if $G$ is connected. Similarly, the generalized spectrum for matrices involving the distance is only defined for graphs in $\CG_n$ and $\mathcal{G}_n^{\GIN}(M)\subseteq \CG_n$ (namely, the distance matrix, the distance (signless) Laplacian, the (signless) adjacency-transmission, and the (signless) degree-distance matrices). Moreover, $| \CG_n|$ denotes the number of connected simple graphs ($| \G_n|$) minus the number of disconnected simple graphs on $n$ vertices.
Table \ref{Tab:generalisedcoinvariantALQD} shows computational results for generalized $M$-coinvariant graphs.

It is known that starting from $n=6$ every connected graph has $A$-coinvariant mates \cite{enumeration}. This is no longer the case for the number of connected graphs with generalized $A$-coinvariant mates (about $85\%$ for $n=6$). However, the data clearly indicates that almost every connected graph has such a mate, for instance for $n=10$ about $99.9\%$ of graphs have generalized $A$-coinvariant mates. Similarly, for the distance matrix, most connected graphs, up to 10 vertices, have a (generalized) $D$-coinvariant mate. On the other hand, there are much fewer generalized $L$-coinvariant and generalized $Q$-coinvariant graphs than $L$-coinvariant and $Q$-coinvariant graphs.  

By Theorem \ref{original-adjacency}, two graphs are generalized $A$-cospectral if and only if they are $(A,\mathbb{R})$-cospectral (similarly for $L$ and $Q$, as seen in Section \ref{sec:characterization} and illustrated in Table~\ref{Tab:generalisedcospectralALQD}).
The number of generalized $A$-cospectral graphs and the number of $L$-cospectral graphs have been previously computed for graphs up to 12 vertices  \cite{bs,HS2004}. On the other hand, for the distance matrix, there are 67 $D$-cospectral graphs that are not $(D,\mathbb{R})$-cospectral (13 with 9 vertices and 54 with ten). For the distance signless Laplacian, there are only 8 cospectral graphs (all of them with 10 vertices) that are not $(D^Q,\mathbb{R})$-cospectral.
For the degree-distance matrix, there are only 8 cospectral graphs that are not $(D^{\deg},\mathbb{R})$-cospectral (all with 10 vertices).
For the transmission-adjacency matrix, there are 235 cospectral graphs that are not $(A^{\trs},\mathbb{R})$-cospectral (10 with 9 vertices and 225 with 10 vertices). Also let us note that up to 10 vertices, if a graph is $D^{\deg}_+$-cospectral or $A^{\trs}_+$-cospectral, then it is also $(D^{\deg}_+,\mathbb{R})$-cospectral or $(A^{\trs}_+,\mathbb{R})$-cospectral, respectively. This computational result relates with what we have seen in Section \ref{sec:characterization}; for graphs having $\trs(G)+\deg(G) =s$ for some integer $s\in \mathbb{Z}$ as in Lemma~\ref{lemma:r1}, $M$-cospectrality for $D^{\deg}_+$ or $A^{\trs}_+$ is equivalent to generalized cospectrality, respectively.

\begin{table}[ht]
    \centering
    \begin{tabular}{|l||c|c|c|c|c|c|c|}
        \hline
        $n$ & 4 & 5 & 6 & 7 & 8 & 9 & 10  \\
        \hline
        $| \mathcal{G}_n |$ & 6 & 21 & 112 & 853 & 11117 & 261080 & 11716571\\
        \hline
        $|\mathcal{G}_n^{\GIN}(A)|$ & 0 & 12 & 95 & 830 & 11079 & 261021 & 11716497 \\
        $|\mathcal{G}_n^{\GIN}(L)|$ & 0 & 0 & 0 & 14 & 886 & 22124 & 950291 \\
        $|\mathcal{G}_n^{\GIN}(Q)|$ & 0 & 2 & 42 & 122 & 1000 & 10467 & 450816 \\
        \hline
        \hline
        $|\CG_n |$ & 1 & 8 & 68 & 662 & 9888 & 247492 & 11427974\\
        \hline
        $|\mathcal{G}_n^{\GIN}(D)|$ & 0 & 0 & 12 & 340 & 7467 & 232611 & 11316322 \\
        $|\mathcal{G}_n^{\GIN}(D^L)|$ & 0 & 0 & 0 & 0 & 45 & 2114 & 185406\\
        $|\mathcal{G}_n^{\GIN}(D^Q)|$ & 0 & 0 & 0 & 0 & 18 & 891 & 78208\\
        $|\mathcal{G}_n^{\GIN}(A^{\trs})|$ & 0 & 0 & 0 & 0 & 32 & 616 & 87841\\
        $|\mathcal{G}_n^{\GIN}(A^{\trs}_+)|$ & 0 & 0 & 0 & 0 &  36 & 2206 & 179094\\
        $|\mathcal{G}_n^{\GIN}(D^{\deg})|$ & 0 & 0 & 0 & 0 & 48 & 964 & 98588\\
        $|\mathcal{G}_n^{\GIN}(D^{\deg}_+)|$ & 0 & 3 & 4 & 34 &  500 & 7915 & 427394\\
        \hline
    \end{tabular}
    \caption{Number of connected generalized $M$-coinvariant graphs with at most $10$ vertices.}
	\label{Tab:generalisedcoinvariantALQD}
\end{table}

Table~\ref{Tab:generalisedcospectralALQD} shows the number of generalized $M$-cospectral graphs up to 10 vertices. We see, for instance, that $D^L$ and $A^{\trs}_+$ are the only matrices capable of distinguishing every graph in $CG_n$ for $n\leq 7$ either by their generalized $M$-spectrum or by their generalized SNF. However, considering $n\leq 10$ the matrix with the least number of generalized cospectral mates is $A^{\trs}$ followed by $D^Q$, and $D^{\deg}$. On the other hand, the matrix with the least number of generalized coinvariant graphs is $D^Q$, followed by $A^{\trs}$, and $D^{\deg}$.

\begin{table}[ht!]
    \centering
    \begin{tabular}{|l||c|c|c|c|c|c|c|}
        \hline
        $n$ & 4 & 5 & 6 & 7 & 8 & 9 & 10  \\
        \hline
        $| \mathcal{G}_n |$ & 6 & 21 & 112 & 853 & 11117 & 261080 & 11716571 \\
        \hline
        $|\mathcal{G}_n^{\GSP}(A)|$ & 0 & 0 & 0 & 32 & 1042 & 41212& 2338933 \\
        $|\mathcal{G}_n^{\GSP}(L)|$ & 0 & 0 & 4 & 115 & 1611 & 40560& 1367215\\
        $|\mathcal{G}_n^{\GSP}(Q)|$ & 0 & 2 & 10 & 80 & 998 & 17453& 613954\\
        \hline
        \hline
        $|\CG_n |$ & 1 & 8 & 68 & 662 & 9888 & 247492 & 11427974\\
        \hline
        $|\mathcal{G}_n^{\GSP}(D)|$ & 0 & 0 & 0 & 0 & 48 & 3480& 276328 \\
        $|\mathcal{G}_n^{\GSP}(D^L)|$ & 0 & 0 & 0 & 0 & 105 & 4118& 245140\\
        $|\mathcal{G}_n^{\GSP}(D^Q)|$ & 0 & 0 & 0 & 4 & 86 & 1519 & 95296\\
        $|\mathcal{G}_n^{\GSP}(A^{\trs})|$ & 0 & 0 & 0 & 4 & 56 & 1212& 75364\\
        $|\mathcal{G}_n^{\GSP}(A^{\trs}_+)|$ & 0 & 0 & 0 & 0 & 105 & 3624 & 232962\\
        $|\mathcal{G}_n^{\GSP}(D^{\deg})|$ & 0 & 0 & 0 & 4 & 76 & 2370& 124866 \\
        $|\mathcal{G}_n^{\GSP}(D^{\deg}_+)|$ & 0 & 0 & 0 & 24 & 413 & 11536 & 445738 \\
        \hline
    \end{tabular}
    \caption{Number of connected generalized $M$-cospectral graphs with at most $10$ vertices.}
	\label{Tab:generalisedcospectralALQD}
\end{table}

\subsection{Enumeration results: the case of diameter $2$}

We also present computational results for graphs $G$ such that $G$ and $\overline{G}$ have diameter 2, up to 11 vertices. Henceforth, let us denote the set of such graphs on $n$ vertices as $\mathcal{GC}_{n,d2}$. Recall that  two graphs $G,H\in \mathcal{GC}_{n,d2}$ satisfy Corollary~\ref{coro:complementdiam2ALQ}. Thus, when we focus on generalized $M$-cospectral in $\mathcal{GC}_{n,d2}$, it is sufficient to compute the sets of generalized $M$-cospectral graphs for the adjacency matrix, the Laplacian and signless Laplacian matrix. Let $|\mathcal{GC}_{n,d2}^{\GSP}(M)|$ denote the graphs in $\mathcal{GC}_{n,d2}$ having at least one generalized $M$-cospectral mate. Similarly, let $|\mathcal{GC}_n^{\GIN}(M)|$ denote the graphs in $\mathcal{GC}_{n,d2}$ having at least one generalized $M$-coinvariant mate.

Table \ref{Tab:generalisedcoALQdiam2}
shows the number of generalized $M$-cospectral and generalized $M$-coinvariant graphs for the aforementioned matrices. We observe that up to 11 vertices, the generalized spectrum of $L$ and $Q$ perform better than the generalized $A$-spectrum in  distinguishing graphs in $\mathcal{GC}_{n,d2}$. Also, it is clear that there are still many graphs with at least one generalized $A$-coinvariant mate. Moreover, Figure \ref{fig:generalizedcoALQdiam2} illustrates the generalized $M$-cospectral uncertainty and the generalized $M$-coinvariant uncertainty of graphs in $\mathcal{GC}_{n,d2}$ for $M=\{A,L,Q\}$ (except for the generalized $A$-coinvariant uncertainty for the aforementioned reasons).

\begin{table}[ht!]
    \centering
    \begin{tabular}{|l||c|c|c|c|c|c|}
        \hline
        $n$ & 6 & 7 & 8 & 9 & 10 & 11  \\
        \hline
        \hline
        $|\mathcal{GC}_{n,d2}|$ & 2 & 18 & 218 & 6069 & 364270 & 44343606\\
        \hline
        \hline
        $|\mathcal{GC}_{n,d2}^{\GSP}(A)|$ & 0 & 0 & 0 & 420 & 48992 & 6935002 \\
        $|\mathcal{GC}_{n,d2}^{\GSP}(L)|$ & 0 & 0 & 23 & 952 & 60884 & 4849676\\
        $|\mathcal{GC}_{n,d2}^{\GSP}(Q)|$ & 0 & 0 & 2 & 212 & 20710 & 1918758\\
        \hline
        \hline
        $|\mathcal{GC}_{n,d2}^{\GIN}(A)|$ & 0 & 10 & 163 & 5918 & 363834 &  44342414\\
        $|\mathcal{GC}_{n,d2}^{\GIN}(L)|$ & 0 & 0 & 9 & 382 & 45250 & 2466748 \\
        $|\mathcal{GC}_{n,d2}^{\GIN}(Q)|$ & 0 & 0 & 0 & 84 & 18760 & 902038\\
        \hline
    \end{tabular}
    \caption{Number of generalized $M$-cospectral  and generalized $M$-coinvariant graphs (with diameter 2 for themselves and their complement) with at most $11$ vertices and $M=\{A,L,Q\}$.}
	\label{Tab:generalisedcoALQdiam2}
\end{table}

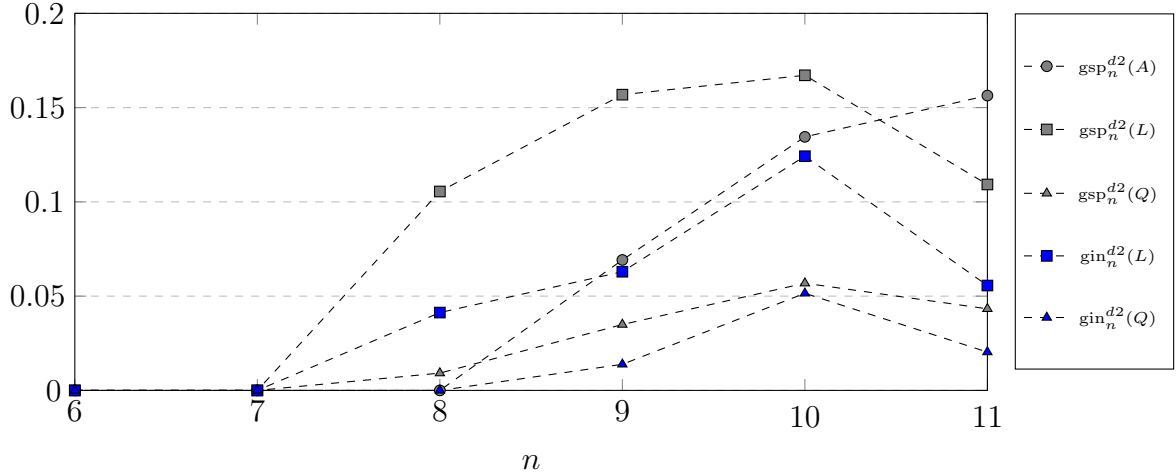
\begin{figure}[ht]
    \centering
    \begin{tikzpicture}[trim axis left]
    
    \begin{axis}[
        scale only axis,
        title={},
        xlabel={$n$},
        width=\textwidth-4cm, 
        height=5cm,
        xmin=6, xmax=11,
        ymin=0.0, ymax=0.2,
        xtick={6,7,8,9,10,11},
        legend style ={ 
            row sep=-0.4cm,
            at={(1.03,1)}, 
            anchor=north west,
            draw=black, 
            fill=white,
            align=left
        },
        ymajorgrids=true,
        grid style=dashed,
        legend columns=1
    ]
     
    \addplot[
        dashed,
        every mark/.append style={solid, fill=gray},
        mark=*,
        legend entry={\tiny $\GSP_n^{d2}(A)$}
        ]
        coordinates {
        (6,0.000000000000000)(7,0.0)(8,0.0)(9,0.06920415)(10,0.134493645)
        (11, 0.15639237819314919945842924907821)
        };
     
    \addplot[
        dashed,
        every mark/.append style={solid, fill=gray},
        mark=square*,
        legend entry={\tiny $\GSP_n^{d2}(L)$}
        ]
        coordinates {
        (6,0.0)(7,0.0)(8,0.105504587)(9,0.156862745)(10,0.167139759)
        (11,0.109230539)
        };

\addplot[
        dashed,
        every mark/.append style={solid, fill=gray},
        mark=triangle*,
        legend entry={\tiny $\GSP_n^{d2}(Q)$}
        ]
        coordinates {
        (6,0.0)(7,0.0)(8,0.00917431193)(9,0.0349316197)(10,0.0568534329)
        (11,0.0432702293088207576)
        };
        
        
    \addplot [
        dashed,
        every mark/.append style={solid, fill=blue},
        mark=square*,
        legend entry={\tiny $\GIN_n^{d2}(L)$}
        ]
        coordinates {
        (6,0.0)(7,0.0)(8,0.04128440366972477064220183486239)(9,0.06294282418849892898335804910199)(10,0.12422104482938479699124275949159)(11, 0.05562804251868916569392214065766)
        };

        \addplot [
        dashed,
        every mark/.append style={solid, fill=blue},
        mark=triangle*,
        legend entry={\tiny $\GIN_n^{d2}(Q)$}
        ]
        coordinates {
        (5,0.0)(6,0.0)(7,0.0)(8,0.0)(9,0.01384083044982698961937716262976)(10,0.05150026079556373020012627995717)(11,0.02034200827059486321432677351499)
        };



    
        
    \end{axis}
\end{tikzpicture}
    \caption{The fraction of connected graphs on $n$ vertices with diameter 2 and complement of diameter 2 that have at least one generalized $M$-cospectral mate with respect to a matrix $M$ is denoted as $\GSP_n^{d2}(M)$, and similarly for $\GIN_n^{d2}(M)$.}
    \label{fig:generalizedcoALQdiam2}
\end{figure}

On the other hand, Table~\ref{Tab:generalisedcoinvdiam2} shows the number of graphs in $\mathcal{GC}_{n,d2}^{\GIN}(M)$ for the graph matrices involving distance ($D$, $D^L$, $D^Q$, $D^{\deg}$, $D^{\deg}_+$, $A^{\trs}$, and $A^{\trs}_+$).
As for the adjacency matrix, we should note that for the distance matrix, most graphs up to 11 vertices have a generalized $D$-coinvariant mate. 
We also observe that the number of generalized coinvariant graphs for the Laplacian and signless Laplacian is smaller than the number of generalized cospectral graphs for such matrices. In this regard, Figure~\ref{fig:generalizedcoALQdiam2} suggests a very similar behavior. Further enumeration results would be needed in order to make a reasonable conjecture regarding this phenomenon.

\begin{table}[ht!]
    \centering
    \begin{tabular}{|l||c|c|c|c|c|c|}
        \hline
        $n$ & 6 & 7 & 8 & 9 & 10 & 11  \\
        \hline
        \hline
        $|\mathcal{GC}_{n,d2}|$ & 2 & 18 & 218 & 6069 & 364270 & 44343606\\
        \hline
        \hline
        $|\mathcal{GC}_{n,d2}^{\GIN}(D)|$ & 0 & 4 & 126 & 5206 & 353826 & 44245420 \\
        $|\mathcal{GC}_{n,d2}^{\GIN}(D^L)|$ & 0 & 0 & 9 & 428 & 45186 & 2615994\\
        $|\mathcal{GC}_{n,d2}^{\GIN}(D^Q)|$ & 0 & 0 & 0 & 84 & 19048 & 932632\\
        \hline
        $|\mathcal{GC}_{n,d2}^{\GIN}(D^{\deg})|$ & 0 & 0 & 0 & 96 & 19280 & 953406 \\
        $|\mathcal{GC}_{n,d2}^{\GIN}(D^{\deg}_+)|$ & 0 & 0 & 110 & 1523 & 116854 & 3495822\\
        $|\mathcal{GC}_{n,d2}^{\GIN}(A^{\trs})|$ & 0 & 0 & 0 & 84 & 18872 & 945612\\
        $|\mathcal{GC}_{n,d2}^{\GIN}(A^{\trs}_+)|$ & 0 & 0 & 8 & 492 & 45544 & 2463526\\
        \hline
    \end{tabular}
    \caption{Number of generalized $M$-coinvariant graphs (with diameter 2 for themselves and their complement) with at most $11$ vertices and $M=\{D,D^L,D^Q,D^{\deg},D^{\deg}_+,A^{\trs},A^{\trs}_+\}$.}
	\label{Tab:generalisedcoinvdiam2}
\end{table}


\section{Concluding remarks}

In this paper, we focused on the generalized spectrum and the generalized invariants of several matrices associated with a graph, and we investigated their power as invariants to distinguish and characterize graphs. Motivated by Stanley's recent interview \cite{interview}, we also presented some new combinatorial applications of the SNF in combinatorics.

In Section \ref{sec:characterization} we showed a new characterization of generalized cospectral graphs in terms of codeterminantal graphs, extending a classic result by Johnson and Newman \cite{JohnsonN}. As an intermediate result, we obtained a necessary and sufficient condition for graphs to be codeterminantal on $\mathbb{Q}[x]$.

In Section \ref{sec:SNF}, we computed the SNF of several graph classes (like trees and complete multipartite graphs) for the matrices $A^{\trs}$ and $A^{\trs}_+$, and we showed that star graphs are determined by the SNF of $A^{\trs}$ and $A^{\trs}_+$. We conjecture that complete multipartite transmission regular graphs are determined by the SNF of $A^{\trs}$ and $A^{\trs}_+$.
By computing the SNF of trees for $A^{\trs}$ and $A^{\trs}_+$, we shed some light on the question of whether almost all trees are determined by the SNF of the aforementioned transmission-adjacency matrices, see \cite{degreedistance}.

The computational results from Section \ref{sec:enumeration} suggest that the generalized spectrum of $A^{\trs}$ may be the best invariant to distinguish graphs (note that generalized $A^{\trs}$-invariants also present a good behavior), followed by the generalized $D^Q$-invariants. For the other considered matrices, it seems that the generalized invariants perform better than the generalized spectrum for distinguishing graphs. For the case of graphs of diameter $2$ and up to $11$ vertices, the generalized $Q$-invariants outperform the rest of the studied invariants.




\subsection*{Acknowledgements}
Aida Abiad is partially supported by the Dutch Research Council through the grant VI.Vidi.213.085 and by the Research Foundation Flanders through the grant 1285921N. The research of Carlos Alfaro is partially supported by CONACyT and SNI.


\end{document}